\documentclass[a4paper,11pt,leqno,english]{smfart}
\usepackage{aeguill}
\usepackage{enumerate}
\usepackage{amssymb,amsmath,latexsym,amsthm}
\usepackage[T1]{fontenc}
\usepackage{geometry}
\usepackage{url}
\usepackage[frenchb, main=english]{babel}
\usepackage[utf8]{inputenc}
\usepackage{mathrsfs}
\usepackage{xcolor}
\usepackage{comment}
\definecolor{violet}{rgb}{0.0,0.2,0.7}
\definecolor{rouge2}{rgb}{0.8,0.0,0.2}
\usepackage{tikz}
\usepackage{empheq}
\usepackage{tikz-cd}
\usetikzlibrary{matrix,arrows,decorations.pathmorphing}
\usepackage{hyperref}
\usepackage{mathpazo}
\hypersetup{
    bookmarks=true,         
    unicode=false,          
    pdftoolbar=true,        
    pdfmenubar=true,        
    pdffitwindow=false,     
    pdfstartview={FitH},    
    pdftitle={},    
    pdfauthor={},     
    colorlinks=true,       
   linkcolor=rouge2,          
    citecolor=violet,        
    filecolor=black,      
    urlcolor=cyan}           
\usepackage{enumitem}
\usepackage{appendix}

\theoremstyle{plain}    
\newtheorem{thm}{Theorem}[section]
\theoremstyle{plain} 
\newtheorem{bigthm}{Theorem}
\newtheorem{bigcoro}[bigthm]{Corollary}
 
 \numberwithin{equation}{section} 
 \numberwithin{figure}{section} 
 \newtheorem{cor}[thm]{Corollary} 
 \theoremstyle{plain}   
 \newtheorem{ass}[thm]{Assumption} 
 \theoremstyle{plain}  
 \newtheorem{prop}[thm]{Proposition} 
 \theoremstyle{plain}    
 \newtheorem{lem}[thm]{Lemma} 
 \theoremstyle{remark}
  \newtheorem{claim}[thm]{Claim} 
 \theoremstyle{remark}
 \newtheorem{rem}[thm]{Remark}
 \theoremstyle{definition}

\theoremstyle{plain}  

\theoremstyle{plain}

\theoremstyle{definition}
\newtheorem{defi}[thm]{Definition}

\newcommand{\C}{{\mathbb{C}}}
\newcommand{\N}{{\mathbb{N}}}

\newcommand{\Q}{{\mathbb{Q}}}
\newcommand{\R}{{\mathbb{R}}}

\newcommand{\bD}{{\mathbb{D}}}

\newcommand{\cG}{{\mathcal{G}}}

\newcommand{\cX}{{\mathcal{X}}}

\def\1{\mathbf{1}}

\newcommand{\wX}{\widehat{X}}

\newcommand{\om}{\omega}

\newcommand{\Xr}{X_{\rm reg}}

\newcommand{\ep}{\varepsilon}

\newcommand{\Ric}{\mathrm{Ric} \,}

\newcommand{\reg}{\mathrm{reg}}

\renewcommand{\ge}{\geqslant}
\renewcommand{\le}{\leqslant}

\newcommand{\tr}{\operatorname{tr}}

\newcommand{\supp}{\operatorname{supp}}

\newcommand{\PSH}{\operatorname{PSH}}

\newcommand{\HBC}{\operatorname{HBC}}

\title{A note on orbifold regularity of canonical metrics}
\date{\today}

\author{Henri Guenancia}
\address{Univ. Bordeaux, CNRS, Bordeaux INP, IMB, UMR 5251, F-33400 Talence, France}\email{henri.guenancia@math.cnrs.fr}

\author{Chung-Ming Pan}
\address{Centre interuniversitaire de recherches en géométrie et topologie (CIRGET); Université du
Québec à Montréal; Case postale 8888, Succursale centre-ville, Montréal, Québec, H3C 3P8, Canada}
\email{pan.chung\_ming@uqam.ca}

\author{Mihai P\u{a}un}
\address{Universit\"at Bayreuth, Mathematisches Institut, Lehrstuhl Mathematik VIII, Universit\"atsstrasse 30, D-95447, Bayreuth, Germany}
\email{mihai.paun@uni-bayreuth.de}

\begin{document}
 
\begin{abstract}
In this short note, we prove that on a compact Kähler variety $X$ with log terminal singularities and $c_1(X)=0$, any singular Ricci-flat Kähler metric has orbifold singularities in restriction to the orbifold locus of $X$. 
\end{abstract}

\maketitle
 
 \tableofcontents
 \section{Introduction}

Let $(X, \omega_X)$ be a compact normal Kähler space with log terminal singularities. It was proved in \cite{EGZ09} that given any smooth hermitian metric $h$ on $K_X$, there exists a unique function $\varphi \in \mathrm{PSH}(X, \omega_X) \cap L^{\infty}(X)$ solving the complex Monge-Ampère equation
\begin{equation}
\label{MAI}
(\omega_X+dd^c \varphi)^n=\mu_h, \quad \sup_X \varphi=0
\end{equation}
where $\mu_h$ is the normalized canonical measure associated with $h$, cf \eqref{mu h}. Moreover, the closed positive $(1,1)$-current $\omega:=\omega_X+dd^c \varphi$ is a genuine Kähler metric in restriction to the regular locus $X_{\rm reg}$ of $X$. \\

In general, the behavior of $\omega$ near $X_{\rm sing}$ is quite mysterious, with the notable exception of when the singularities of $X$ are only ordinary double points \cite{HS}, cf also \cite{CS} for further results. A much more amenable problem is to consider the behavior of $\omega$ near finite quotient singularities (also called orbifold singularities), i.e. near points $x\in X_{\rm sing}$ where $(X,x)$ is locally analytically isomorphic to a finite quotient of $\mathbb C^n$. Let us define $X^{\rm orb}$ to be the maximal euclidean open subset of $X$ with at most quotient singularities. It is expected that $\omega|_{X^{\rm orb}}$ has orbifold singularities, i.e. that the pull back of $\omega$ to any smooth local cover is a Kähler metric. The problem is however not as simple as it sounds because of the global nature of $\omega$. Here is what is known so far: 
\begin{enumerate}[label= $\bullet$]
\item If $X$ has only orbifold singularities, then $\omega$ has orbifold singularities. This is a pretty straightforward adaptation of the smooth case, cf e.g. \cite{Camp04, Faulk19}. 
\item If $x\in X^{\rm orb}$ is an isolated singularity, then $\omega$ has orbifold singularities near $x$. This follows essentially from \cite{EGZ09}. 
\item  When $X$ is {\it projective}, then $\omega|_{X^{\rm orb}}$ has orbifold singularities. It was proved in \cite{LiTian19} relying on a construction of C. Xu. An alternative proof was later provided in \cite{GP24}, cf Remarks~\ref{rem orbi 1} and \ref{rem orbi 2}. 
\end{enumerate}

\medskip

The goal of this note is to give a "quantitative" version of $\omega$ having orbifold singularities (cf Proposition~\ref{quanorbi}) so as to get the following. 

\begin{bigthm}
\label{thma}
Let $(X,\omega_X)$ be a compact Kähler variety with log terminal singularities, let $h$ be a smooth hermitian metric on $K_X$ and let $\omega=\omega_X+dd^c \varphi $ be the solution of \eqref{MAI}. 
Assume that $X$ admits a locally trivial algebraic approximation. Then $\omega|_{X^{\rm orb}}$ has orbifold singularities. 
\end{bigthm}

We refer to Section~\ref{sec appli} for the precise meaning of the assumption on $X$. \\

When $X$ has trivial first Chern class, then the solution of \eqref{MAI} with respect to a flat hermitian metric on $K_X$ yields the so-called singular Ricci flat metric $\omega\in [\omega_X]$; the latter satisfies $\Ric \omega =0$ on $X_{\rm reg}$ in the usual sense. Now it was proved in \cite{BGL} that such varieties $X$ do admit locally trivial algebraic approximations. Given \cite{LiTian19} and Theorem~\ref{thma} one immediately obtains

\begin{bigcoro}\label{bigcoro:orb_CY}
Let $(X,\omega_X)$ be a compact Kähler variety with log terminal singularities such that $c_1(K_X)=0$. Then the singular Kähler Ricci flat metric $\omega\in [\omega_X]$ has orbifold singularities in restriction to $X^{\rm orb}$. 
\end{bigcoro}

\begin{rem}
If $X$ is a log terminal compact Kähler variety such that $K_X$ is ample (resp. $-K_X$ is ample) and is endowed with a singular Kähler-Einstein metric $\omega$, i.e. $\Ric \om = - \om$ (resp. $\Ric \om= \om$), then since $X$ is projective  the arguments in \cite{LiTian19} and \cite{GP24} carry over to show that $\om|_{X^{\rm orb}}$ has orbifold singularities. Together with the corollary above, this shows that {\it any} singular Kähler-Einstein metric on a log terminal compact Kähler variety $X$ has orbifold singularities in restriction to $X^{\rm orb}$.  
\end{rem}

\begin{rem}
Corollary~\ref{bigcoro:orb_CY} implies that the metric space associated with $(X^{\rm orb}, \omega)$ is locally bi-Hölder to $(X^{\rm orb}, \omega_X)$, cf Proposition~\ref{prop:orb_metric_completion} and Corollary~\ref{completion}.
\end{rem}
\bigskip

{\bf Strategy of the proof.}
The proof goes roughly as follows, cf~Theorem~\ref{appl1} for details. Let $\cX\to \mathbb D$ a locally trivial deformation of $X$ such that the fibers $X_{t_k}$ ($t_k\to 0$ as $k\to +\infty$) are projective. We consider the solutions $\omega_t$ of the equation \eqref{MAI} on $X_t$. Given $x\in X^{\rm orb}$ and a small neighborhood $U\subset \mathcal X^{\rm orb}$ of $x$, we pick an orbifold Kähler metric $\omega_{\rm orb}$ on $U$. The name of the game is to uniformly estimate the function 
\[f_t:=\mathrm{tr}_{\omega_t}({\omega_{\rm orb}}|_{U_t})\]
on $U_t$ for $t=t_k$ as $k\to+\infty$. The key steps are the following: 
\begin{enumerate}
\item Small powers $\alpha$ of $f_t$ satisfies an elliptic inequality $\Delta_{\omega_t} f_t^\alpha \ge -C_t f_t^{\alpha}$. 
\item $C_t$ can be chosen independently of $t$. This relies on the recent results \cite{G+} yielding strict positivity for solutions of MA equations like \eqref{MAI}, cf Lemma~\ref{quanstrikt} for the quantitative version. In order to go further one needs to integrate the elliptic inequality, which requires a cutoff function $\chi$ with support in $U$ as well as a family of cutoff functions $\rho_\ep$ for $U_{\rm sing}\subset U$. 
\item One can find $\chi$ such that $|d \chi|_{\omega_t} \le C$ using the arguments from the previous step. Next, since $\omega_t$ has orbifold singularities (at this point only in a {\it qualitative} way), one can arrange that $\int_{U_t} |d\rho_\ep|^4 \omega_t^n \to 0$ as $\ep \to 0$. 
\item One can then appeal to the Harnack type inequality from \cite{GP24} (itself building upon the results of \cite{GPSS22}) to derive from the easy $L^1$ bound for $f_t$ an actual $L^{\infty}$ bound. This is the object of Proposition~\ref{quanorbi}. 
\item Once a uniform bound is obtained for $\sup_{U_{t_k}} f_{t_k}$, we rely on well-established pluripotential theoretic arguments to show that $\sup_{U_0} f_0<\infty$, which essentially concludes the proof. 
\end{enumerate}

\bigskip

{\bf Acknowledgements.} H.G. is partially supported by the French Agence Nationale de la Recherche (ANR) under reference ANR-21-CE40-0010 (KARMAPOLIS). 
C.-M.P. is supported by postdoc funding from CIRGET, UQAM. 
M.P. is gratefully acknowledging the support from the Deutsche Forschungsgemeinschaft (DFG).

 \section{Preliminaries} \label{prel}

Let $X$ be a compact normal Kähler variety. 

\begin{defi}[Kähler forms]
\label{def1}
A Kähler form on $X$ is a Kähler form $\omega_X$ on $\Xr$ such that $\omega_X$ extends to a Kähler form on $\mathbb C^N$ given any local embedding $X\underset{\rm loc}{\hookrightarrow} \mathbb C^N$. 
\end{defi}

Assume that the rank one reflexive sheaf $mK_X:=((\det \Omega_X^1)^{\otimes m})^{\star\star}$ is locally free for some integer $m\ge 1$. It makes sense to consider a smooth hermitian metric $h$ on $K_X$ as well as a local generator $\sigma$ of $mK_X$. Then  
\begin{equation}
\label{mu h}
\mu_h:=i^{n^2} \frac{(\sigma\wedge \bar \sigma)^{\frac 1m}}{|\sigma|^{2/m}_{h^{\otimes m}}}
\end{equation}
defines a positive measure on $\Xr$ which is independent of the choice of $m$ or $\sigma$. We extend it to $X$ trivially. We have $\mu_h(X)<+\infty$ if and only if $X$ has log terminal singularities. \\

\noindent {\bf Convention} From now on, we further assume that $X$ has log terminal singularities and dimension equal to $n$. 
\smallskip

\begin{defi}[Orbifold singularities]
Let $x\in X$; we say that the germ $(X,x)$ is a quotient (or orbifold) singularity if there exist a euclidean open neighborhood $U$ of $x$, and a finite quasi-\'etale Galois cover 
\begin{equation}\label{Galois cover}
	p: V \to U
\end{equation} 
such that $V$ is an open subset of $\mathbb{C}^n$. 
\end{defi}
In the sequel, for simplicity, we may assume that $V$ is the unit ball $B_{\C^n}(0,1)$ in $\C^n$ and set $U_r := p(B_{\C^n}(0,r))$ for each $r \in (0,1]$. 

\begin{defi}[Orbifold locus]
We define the orbifold locus of $X$ by
\[
	X^{\rm orb} := \{ x \in X \mid (X,x) \text{ is at most quotient}\},
\]
which is a euclidean open set of $X$.
\end{defi}
According to \cite[Proposition 9.3]{GKKP} and \cite[Lemma 5.8]{GK20}, if $X$ is log terminal, then $X$ is quotient in codimension two; namely, the non-orbifold locus $X \setminus X^{\rm orb}$ is contained in a proper analytic subset of codimension at least three. 
Moreover, when $X$ is projective or admits a locally trivial algebraic approximation, it is known that $X \setminus X^{\rm orb}$ is indeed a closed analytic subset (see \cite[Lemma~35]{CGG2}).

\smallskip

The Monge-Ampère equation plays a key role in the study of complex Kähler spaces. Its importance stems from the fact that it offers a quantitative counterpart to numerical properties of the canonical class of a Kähler space. That is to say, one can construct metrics whose Ricci curvature reflects the natural properties arising from the context of algebraic geometry. 

\smallskip
Given a Kähler form $\omega_X$, we  rescale $h$ so that $\mu_h(X)=\int_X\omega_X^n$. We consider the Monge-Ampère equation
\begin{equation} 
\tag{MA}
\label{MA}
(\omega_X+dd^c \varphi)^n=\mu_h,
\end{equation}
for $\varphi \in \PSH(X, \omega_X)\cap L^{\infty}(X)$ normalized so that $\sup_X \varphi=0$. 

\smallskip
It was proved in \cite{EGZ09} that \eqref{MA} admits a unique solution $\varphi$. Moreover, the positive current $\omega:=\omega_X+dd^c \varphi$ induces a Kähler metric on $\Xr$, satisfying
\begin{equation} 
\label{KE}
\Ric \omega=-\Theta_h(K_X).
\end{equation}
For an outline of the resolution of \eqref{MA} we refer to \cite[\textsection 2]{GP24}.
\medskip

\noindent In this context, the main question we are considering here is the following: \emph{has the restriction 
\[\omega|_{X^{\rm orb}}\]
orbifold singularities}, in the sense that $p^*\omega|_{U}$ is a Kähler metric on $V$ for any $U,V, p$ as in \eqref{Galois cover}? 

\begin{rem}
\label{rem orbi 1}
It is enough to check that $p^*\omega|_U$ is a Kähler metric for a collection of open sets $U$ as above covering $X^{\rm orb}$. Indeed, if we have two uniformizing charts $p_i: V_i\to U_i$ for $i=1,2$ such that $U_{12}:=U_1\cap U_2 \neq \emptyset$, then the "transition maps" $V_{1}'\times_{U_{12}} V_{2}'\to V_{i}'$ are quasi-étale, where $V_{i}' := p_i^{-1}(U_{12})$. 
Since the $V_i$'s are smooth, these transition maps are actually étale by Nagata purity theorem.  
\end{rem}

\begin{rem}
\label{rem orbi 2}
This is known to be the case if $X$ is projective, thanks to the work of Li and Tian, \cite{LiTian19}. One of the crucial ingredients in their proof is the existence of a \emph{partial desingularisation} of $X$. That is to say, they show that there exists a projective variety $\wX$ such that 
$\wX^{\rm orb}= \wX$ together with a map $\pi: \wX\to X$ biholomorphic above $X^{\rm orb}$. A similar, slightly weaker result was announced by Ou in 
\cite{Ou} for any normal complex space $X$. It is our deliberate choice to rely here exclusively on the published result in \cite{LiTian19}, even if in the end the theorems we are able to establish are less general than expected.
\end{rem}

\begin{rem}
A proof of Li-Tian's result which is independent of the existence of a partial resolution of a projective variety $X$ is given in \cite[\textsection~4]{GP24}. Unfortunately, the arguments provided in the aforementioned work are still too dependent on the fact that $X$ is projective, through the use of Artin's approximation theorem.
\end{rem}


\section{Strict positivity}\label{poz}

To begin with, we recall the following result, due to \cite[Theorem~1.2]{G+}, which is a vast generalization of \cite{GGZ2}.  
\begin{thm}\cite{G+}
\label{strikt}
Let $(X, \omega_X)$ be a compact normal Kähler space, endowed with a smooth Kähler metric $\omega_X$ (cf. Definition \ref{def1}). We consider 
$\omega:= \omega_X+ dd^cu$ a (closed) positive current on $X$ with the following properties.
\begin{enumerate}
\smallskip

\item[(1)] The potential $u$ is bounded and smooth at least on $\Xr\setminus D$, where $D$ is a divisor on $X$.
\smallskip

\item[(2)] We have $\omega^n= \exp(F)\omega_X^n$ on $\Xr\setminus D$, such that $F\in L^1(X, \omega_X)$ and $\exp(F)\in L^p(X, \omega_X)$ for some $p> 1$.
\smallskip

\item[(3)] $\Ric \omega\geq -A(\omega+ \omega_X)$ in the sense of currents on $\Xr
$, where $A> 0$ is a positive constant.
\end{enumerate}
Then there exists a positive constant $C> 0$ such that $\omega\geq C\omega_X$. Moreover, $C$ only depends on $(X, \omega_X), A$ 
and the $L^p$ norm of $\exp(F)$. 
\end{thm}

\noindent The proof of Theorem \ref{strikt} consists in two steps: the heart of the matter is to show that one can approximate the inverse image of the current $\omega$ on any smooth model of $X$ by a family $(\omega_k)_{k\geq 1}$ of Kähler metrics, in such a way that the negative part of 
$\Ric {\omega_k}$ has a uniform lower bound, i.e. independent of $k$. Then the result is a direct consequence of Chern-Lu inequality, combined with the Green kernel estimates in \cite{GPSS22}. 

\begin{rem}
Straightforward modifications of the arguments given in \cite{G+} show that one can derive the same conclusion as in Theorem \ref{strikt}
if the hypothesis (3) above is replaced by the more general 
\[\Ric\omega+ dd^c\psi \geq -A(\omega+ \omega_X) \]
in the sense of currents on $\Xr$, where $\psi$ is bounded, such that $dd^c\psi \geq - C\omega_X$ for some positive constant $C>0$.
\end{rem}
\medskip

\noindent As a consequence of Theorem \ref{strikt} we infer the following statement, which in fact represents a quantitative version of it.

\begin{lem}\label{quanstrikt} Let $(X, \omega_X)$ be a compact Kähler space with log terminal singularities. Let $h$ be a smooth metric on the $\Q$-line bundle $K_X$, and let $\varphi\in \PSH(X, \omega_X)\cap L^{\infty}(X)$ be the solution of the MA equation
\begin{equation} 
\label{MAF}
(\omega_X+dd^c \varphi)^n= \mu_h,
\end{equation}
normalized so that $\sup_X \varphi=0$. Here $h$ is normalized so that $\mu_h(X)=\int_X\omega_X^n$. Let $A, B$ and $C$ be three positive constants, such that we have:
\begin{itemize}

\item[\rm (1)] $i\Theta_h(K_X)\leq A\omega_X$ on $X$;

\item[\rm (2)] The holomorphic bisectional curvature $\HBC_{\omega_X}$ is bounded from above on $\Xr$ by $B$;

\item[\rm (3)] $\varphi\geq -C$.
\end{itemize}
Then there exists $\ep_0= \ep_0 (A, B, C,n)> 0$ such that $\omega\geq \ep_0 \omega_X$.
\end{lem}

\begin{rem}
In the statement above, we do not require a $L^p$ bound ($p>1$) for the density of $\mu_h$ but rather we ask for an $L^{\infty}$ bound of the potential $\varphi$. 
\end{rem}

\begin{proof} The proof is an \emph{énième} consequence of the Chern-Lu inequality, as follows. In the first place, by MA theory we already know that 
$\omega= \omega_X+dd^c \varphi$ is non-singular on $\Xr$. Moreover, Theorem \ref{strikt} shows that we have
\[\sup_{\Xr}\tr_\omega \omega_X< \infty.\]
Let $\psi\in \PSH(X, \omega_X)$ be a function with log-poles on $X$, such that the polar locus $(\psi= -\infty)$ coincides with $X_{\rm sing}$, normalized by $\sup_X \psi =0$. For each parameter $\delta> 0$ we introduce the function
\[H_\delta:= \delta\psi+ \log\tr_\omega \omega_X\]
defined and smooth on $\Xr$. Thanks to Theorem~\ref{strikt}, we have $H_{\delta}\to -\infty$ near $X_{\rm sing}$ hence $H_\delta$ attains its maximum at a point $x_\delta \in X_{\rm reg}$. 

\noindent As a consequence of the Chern-Lu inequality we have
\begin{equation}\label{CL2}
	\Delta_\omega \log \tr_\omega \omega_X\geq \frac{\langle \Ric \omega, \omega_X\rangle_{\omega}}{\tr_\omega \omega_X}- B\tr_\omega \omega_X.
\end{equation}

Hypothesis (1) in our statement of Lemma \ref{quanstrikt} can be re-written as $\Ric \omega \geq -A\omega_X$. 
Note that for any two $(1,1)$-forms $\alpha \geq 0$ and $\beta \geq 0$, one has $\langle \alpha, \beta\rangle_\omega \leq (\tr_\omega \alpha)(\tr_\omega \beta)$. Therefore we obtain 
\begin{equation}\label{CL3}
\Delta_\omega \log \tr_\omega \omega_X
\geq -(A+ B)\tr_\omega \omega_X,
\end{equation}
pointwise on $\Xr$. This implies  the following inequality
\begin{equation}\label{CL4}
\Delta_\omega H_\delta
\geq -(A + B + \delta)\tr_\omega \omega_X
= -(A + B + \delta)\exp(H_0) 
\end{equation}
from which we infer that 
\begin{equation}\label{CL5}
\Delta_\omega \big(H_\delta- (A+B+ 2)\varphi\big)
\geq \exp(H_0) -n(A+B+2)
\end{equation}
holds, for any $0<\delta< 1$. By definition of $x_\delta$, we have 
\begin{equation}\label{CL6}
H_\delta(x)- (A+B+ 2)\varphi(x)\leq H_\delta(x_\delta)- (A+B+ 2)\varphi(x_\delta)
\end{equation}
for any $x\in \Xr$ as well as
\begin{equation}\label{CL7}
H_0(x_\delta)\leq \log\big(n(A+B+2)\big).
\end{equation}
as an application of \eqref{CL5} and the maximum principle. Since $\psi\le 0$ and $-C\le \varphi\le 0$ we infer from \eqref{CL6} and \eqref{CL7} that the inequality
\begin{equation}\label{CL8}
H_\delta(x)\leq \log\big(n(A+B+2)\big)+ C(A+B+2)
\end{equation}
holds for any $x\in \Xr$, and our statement is completely proved as $\delta\to 0$.
\end{proof}

\section{Orbifold singularities}\label{osing}

\noindent A glance at the equation \eqref{MAF} shows that the RHS becomes \emph{non-singular} on the Euclidean unit ball $V$ via the 
pull-back by the ramified cover $p: V\to U$, cf. \eqref{Galois cover} in Section \ref{prel}. One is therefore entitled to expect that the restriction of the solution $\omega$ of this equation to $X^{\rm orb}$ has orbifold singularities, that is to say 
\begin{equation}\label{O1}
p^\star \omega|_{U_{\rm reg}}
\end{equation}
extends to a smooth metric on $V$, for any local uniformisation $p$. 
\medskip

In this section, our goal is as follows. Assume that we already know that the pull-back \eqref{O1} has the property that 
\begin{equation}\label{O2}
\sup_{V_0}\tr_{p^\star \omega}\omega_{V}< \infty,
\end{equation}
where $V_0:= p^{-1}(U_{\rm reg})$ and $\om_V$ is the euclidean metric on $V$. We want to convert \eqref{O2} into a quantitative statement, i.e. to determine an explicit constant $C_U> 0$ such that 
\begin{equation}\label{O3}
\sup_{V_0}\tr_{p^\star \omega}\omega_{V}\leq C_U.
\end{equation}
This will be achieved in Proposition~\ref{quanorbi} below, which is based on the methods and results from \cite{GP24} and also relies on \cite{G+} via Lemma~\ref{quanstrikt}. 

\medskip

To start with, it is standard (e.g. using local embedding in euclidean space) to construct a cutoff function $\chi$ on $U$, such that 
\[\supp \chi \Subset U, \qquad \chi|_{U_{1/2}}\equiv 1, \quad \sup_U \frac{|\nabla \chi|^2_{\omega_X}}{\chi}<+\infty\]
and introduce the quantities
\begin{equation}\label{O4}
D:= \sup_{U}\left(\frac{|\nabla \chi|^2_{\omega_{X}}}{\chi}+ |dd^c\chi|_{\omega_{X}}\right) \quad \mbox{and} \quad E:=\sup_V \mathrm{tr}_{\omega_V} ({p^*\omega_X}|_U).
\end{equation}

Consider the function 
\begin{equation}\label{O5}
\hat f: V\to \R_+, \qquad \hat f= \tr_{p^\star \omega}\omega_{V}
\end{equation}
defined and \emph{smooth} on $V$. Indeed, the fact that $\hat f$ is smooth follows from the assumption \eqref{O2}, combined with the fact that 
$\omega$ is the solution of \eqref{MAF}. More precisely, the latter assumptions imply that $p^*\omega$ is quasi-isometric to $\omega_V$, and then the version of Evans-Krylov theorem proved by Y. Wang in \cite[Theorem~1.1]{WangY} combined with the usual bootstrap shows that $p^*\omega$ is a smooth Kähler metric. 
Moreover, we can assume that 
\[{\rm Gal}(p)\subset \mathrm{U}(n)\]
so that $\hat f=p^*f$ is the pull-back of a function $f$ defined on $U$, smooth in orbifold sense.
\smallskip

\noindent As we will explain below, the calculations in \cite{GP24} combined with Lemma~\ref{quanstrikt} show the following

\begin{claim}
\label{claim0}
Let $\alpha \in (0,1)$. There exists a cutoff function $\widetilde \chi$ on $U$ supported on $U_{3/4}$ such that we have 
\begin{equation}
\label{O6}
\Delta_{\omega}(\chi f^\alpha)\geq -C_0\widetilde\chi f^\alpha,
\end{equation}
where $C_0=C_0(\alpha, A, B, C,D,n)$ is a positive constant.
\end{claim} 

\begin{proof}[Proof of Claim~\ref{claim0}]
The heart of the matter is that we have 
\begin{equation}\label{O14}
\Delta_{p^\star\omega}\log \hat f\geq - A\ep_0^{-1},
\end{equation}
by Chern-Lu inequality, since Lemma~\ref{quanstrikt} yields
 \[\Ric {p^\star \omega}\geq -Ap^{\star}\omega_X\ge -A\ep_0^{-1}\omega\]
and the curvature of the flat metric $\omega_V$ is zero. 

On the other hand, the inequality
\begin{equation}\label{O15}
\sup_{U}(\chi^{-1}|\nabla \chi|^2_{\omega_{X}}+ |dd^c\chi|_{\omega_{X}})\geq {\ep_0} \sup_{U}(\chi^{-1}|\nabla \chi|^2_{\omega}+ |dd^c\chi|_{\omega})
\end{equation}
equally holds, by Lemma \ref{quanstrikt} again.

Then the calculations in Step 2 of the proof of \cite[Thm 5.15]{GP24} show that for any $\alpha \in (0,1)$ there exists a constant $C_\alpha$ depending on $\alpha, A\ep_0^{-1}, D, n$ such that 
\begin{equation}\label{O16}
\Delta_{p^\star\omega}(\chi_V \hat f^\alpha)\geq -C_\alpha \widetilde\chi_V \hat f^\alpha
\end{equation}
holds, where $\chi_V:= \chi\circ p$. Descending this inequality to $U$ yields \eqref{O6}.
\end{proof}
\smallskip

As explained in \cite[\textsection~3]{GP24}, the results proved in \cite{GPSS22} yield the existence of a Green kernel $\cG$ for $(\Xr, \omega)$ along with positive constants $\gamma> 0, G> 0$ such that 
\begin{equation}\label{O8}
\int_{\Xr}\cG_x^{1+ \gamma}dV_\omega< G
\end{equation}
for all $x\in \Xr$.
\medskip

\noindent The next statement follows:

\begin{prop}\label{quanorbi}
There exists a positive constant $C_U$ only depending on $C_0, \gamma, G$ as well as an upper bound for $\displaystyle \int_V\omega_V\wedge p^\star\omega^{n-1}$
such that $f\leq C_U$.
\end{prop}

\begin{proof}
Let $x\in X_{\rm reg}$ and let $p:V\to U$ be as in \eqref{Galois cover}. The inverse image $p^{-1}(U_{\rm sing})$ is the trace on $V\subset \mathbb C^n$ of a finite union of linear subspaces of codimension at least two. It is elementary in that setting to produce a sequence $(\tau_k)$ of smooth functions compactly supported on $p^{-1}(U_\reg)$ which converge to $1$ as $k\to+\infty$ and satisfy
\begin{equation}\label{cut off}
\limsup_{k\to +\infty}\int_{V}(|\nabla \tau_k|_{\omega_V}^{4}+ |\Delta_{\omega_V} \tau_k|^{2})\omega_V^n= 0.
\end{equation}
E.g., if $(f_1=\ldots =f_r=0)$ is the equation of a component of $p^{-1}(U_{\rm sing})$, then the function $\tau_k=\xi_k(\log(-\log |f|^2))$ where $\xi_k\equiv 1$ on $(-\infty, k]$ and $\xi_k\equiv 0$ on $[k+1, +\infty)$ will satisfy $|d\tau_k|^2_{\omega_V}+|dd^c \tau_k|^2_{\omega_V} \le \frac{C}{(-\log |f|^2)|f|^2}$ and it is straightforward to check the integrability property \eqref{cut off} above. In order to take into account all the irreducible components of $p^{-1}(U_{\rm reg})$, we just consider the product of all such cutoff functions for each component. Finally, since $p^{-1}(U_{\rm sing})$ is $G$-invariant, one can average such a cutoff function $\tau_k$ in order to obtain one which is $G$-invariant (i.e. $\tau_k=p^*\rho_k$) and still satisfies \eqref{cut off}. Finally, since $p^*\omega|_U$ is (qualitatively) quasi-isometric to $\omega_V$, we have found a sequence of cutoff functions $\rho_k$ supported on $U_{\rm reg}$ such that
\begin{equation}
\limsup_{k\to +\infty}\int_{U}(|\nabla \rho_k|_{\omega}^{4}+ |\Delta_{\omega} \rho_k|^{2})\omega^n= 0.
\end{equation}
Now one can apply Proposition~3.3 in \cite{GP24} to the function $\chi f^\alpha$ satisfying the elliptic inequality \eqref{O6}. More precisely, we use its quantitative form, i.e. the equation (3.1) in {\it loc. cit.}, which yields
\begin{equation}
\label{borne sup}
f(x)^\alpha \le \frac 1{[\omega_X]^n} \int_X \chi f^\alpha \omega^n + C_\alpha\|\cG_x\|_{L^{1+\gamma}} \|\widetilde \chi f^\alpha\|_{L^{1+\frac 1\gamma}}.
\end{equation}
Finally, let us observe that $\int_U f \omega^n \le \frac{n}{|G|} \int_V \omega_V \wedge p^*\omega^{n-1}$ and $f^\alpha \le 1+f$ since $\alpha \in (0,1)$. Now choose $\alpha=\frac {\gamma}{\gamma+1}$ and the lemma follows from the latter observation coupled with the upper bound \eqref{borne sup}. 
\end{proof}

\smallskip

\noindent Let us end this section by the following observation.

\begin{rem}\label{rem:more_quantitative}
The $L^1$ norm of our function $f= \tr_{p^\star \omega}\omega_{V}$ with respect to the Euclidean metric $\omega_V$ on $V_{1/2}$
is in fact bounded by a constant only depending on $\sup_V \mathrm{tr}_{\omega_V}p^*({\omega_X}|_U)$ and $\displaystyle \Vert \varphi\Vert_{L^\infty}$. This is a direct consequence of the Chern-Levine-Nirenberg inequalities. In conclusion, as soon as we know \emph{a priori} that $\omega$ has orbifold singularities, the constant $C_U$ depends on $n, \gamma, A, B, C, D,E, G$.
\end{rem}

\section{Applications}

\subsection{Orbifold regularity for singular Ricci-flat Kähler metrics}
\label{sec appli}
\noindent Let $(X, \omega_X)$ be a compact Kähler space with log terminal singularities admitting a locally trivial algebraic approximation, i.e. such that the following holds. 

\begin{ass}\label{defor} There exists a locally trivial proper map $p:\cX\to \bD$ over a unit disk $\bD\subset \C$ such that $X= p^{-1}(0)$ is the central fiber, and such that there exists a sequence $(t_k)\subset \bD$ converging to $0$ for which the corresponding fibers $X_k:= p^{-1}(t_k)$ is projective.
\end{ass} 

Up to shrinking $\bD$, the \emph{locally trivial} assumption means that there exists a covering $(\mathcal U_\alpha)$ of $\cX$ such that for each $\alpha$ we have a biholomorphism 
\begin{equation}\label{A0}
\mathcal U_\alpha \to U_{\alpha}\times \bD 
\end{equation}
over the disk, where $U_\alpha:= \mathcal U_\alpha \cap X$. In particular, it follows that each fiber of $p$ has log terminal singularities.
\smallskip

\noindent We collect next a few results which will be used in this section. The first and foremost is due to J. Bingener, cf. \cite[Corollary~4.8 \& Theorem~6.3]{Bingener}, and represents the singular analogue of the Kodaira-Spencer stability theorem.

\begin{thm}\cite{Bingener} Let $p:\cX\to \bD$ be a locally trivial, proper map, such that the central fiber $X$ admits a K\"ahler metric $\omega_X$ and has rational singularities. Then there exists a Hermitian metric $\omega_{\cX}$ whose restriction $\omega_{X_t}:={\omega_{\cX}}|_{X_t}$ to each fiber $X_t= p^{-1}(t)$ is a smooth K\"ahler metric.	
\end{thm}

Actually, the result in \emph{loc. cit.} is more general than the statement above, but it is this version that will be relevant for us here.
\medskip

Since the deformation is locally trivial and $K_X$ is a $\mathbb Q$-line bundle, it follows that the relative canonical bundle $K_{\cX/\mathbb D}$ is a $\mathbb Q$-line bundle. So we can consider $\mu_t$ the canonical measure on $X_t$ induced by a fixed, smooth metric $h$ on $K_{\cX/\bD}$. Since the function $t\mapsto \int_{X_t} \omega_{X_t}^n$ is smooth (as one can see using local triviality), we can rescale $h$ so that $\mu_{t}(X_t)=\int_{X_t}\omega_{X_t}^n$. Now we consider for each $t\in \mathbb D$ the equation
\begin{equation}\label{A1}
(\omega_{X_t}+ dd^c\varphi_t)^n= \mu_t, \qquad \sup_{X_t}\varphi_t= 0
\end{equation}
which we know has a unique solution. \medskip

\noindent By analogy with the smooth case, the following statement is established in \cite[Proposition~6.6]{BGL}, cf also \cite[Corollary~5.5]{PT25}.

\begin{lem} 
\label{ke conv}
Let $\omega_t:= \omega_{X_t}+ dd^c\varphi_t$ be the solution of \eqref{A1}. Then 
\[\omega_t\to \omega_0\]
as $t\to 0$ in the weak sense, and locally smoothly on compact subsets of $\Xr$. Here the convergence is understood via the local identification of fibers, cf. \eqref{A0}.
\end{lem}
\medskip

\noindent As already mentioned, in the projective setting the following result holds.
\begin{thm}\label{orbisI}\cite{LiTian19, GP24} Let $X$ be a projective variety with log terminal singularities. Then the solution of \eqref{MAF} has orbifold singularities on $X^{\rm orb}$.  
\end{thm}

\noindent With this at hand, we have the following statement.
\begin{thm}\label{appl1} Let $(X, \omega_X)$ be a normal compact K\"ahler space, with log terminal singularities and such that Assumption \ref{defor} is satisfied. 
Then the unique solution of \eqref{A1} for $t= 0$ has orbifold singularities when restricted to $X^{\rm orb}$.
\end{thm}
\begin{proof}
Consider the metric $\displaystyle \omega_{t_k}$, solution of \eqref{A1} for $t:= t_k$. Since the ambient space is projective, it follows from Theorem~\ref{orbisI} that $\omega_{t_k}$ has orbifold singularities on $X_{t_k}^{\rm orb}$. 

\noindent The proof will be finished once we will have observed that the following uniformity properties hold. Unless indicated, all the constants which appear are assumed to be independent of $|t|< 1/2$.
\begin{enumerate}

\item[A.] \emph{The Ricci curvature of $\omega_t$ is uniformly bounded from below by $\displaystyle -A\omega_{X_t}$.} Indeed, there is $A>0$ such that $\pm i\Theta(K_{\cX/\mathbb D})\le A\omega_{\cX}$. Restricting the latter estimate to $X_t$ and plugging it in \eqref{A1} yields the claim. 
\item[B.] \emph{The holomorphic bisectional curvature of the restriction
\[\omega_{X_t}|_{X_{t}^{\rm reg}} =\omega_{\cX}|_{X_t^{\rm reg}}\]
is bounded from above by a constant $B$.} This is indeed clear by Griffiths formula which shows that the Chern curvature tensor does not increase when restricted to a subbundle. 
\item[C.] \emph{There exists a positive constant $C> 0$ such that 
\[\sup_{X_t}|\varphi_t|\leq C.\]}
Indeed, this follows essentially from \cite{DGG23}, although the precise statement that we would need is not explicitly stated there. A way around is to consider a simultaneous resolution of singularities $\mathcal Y\to \cX$ of the family $\cX\to \mathbb D$ (cf e.g. \cite[Lemma~4.8]{BL18}) and fix a hermitian form $\omega_{\mathcal Y}$ on $\mathcal Y$. Then it is clear that there is $\ep>0$ such that the $L^{1+\ep}$ norm of the densities of the pull backs of $\mu_t$ to $Y_t$ with respect to $\omega_{Y_t}^n$ is uniformly bounded in $t$. The conclusion follows by applying Theorem~3.4 and Theorem~1.1 in \cite{DGG23}. 
\item[D, E.] The fact that the constants $D,E$ from \eqref{O4} can be chosen uniformly in $t$ is an immediate consequence of the local triviality of the family $\cX\to \mathbb D$. 
\item[$\gamma$, G.]	 \emph{The constants $\gamma, G$ in Section \ref{osing} are uniform.} Given how the functions $\mathcal G_x$ are constructed (cf \cite[\textsection~3]{GP24}), the claim is a consequence of the family (and upgraded) version of \cite{GPSS22} proved in \cite[Theorem~A]{GuedjTo} applied to a simultaneous resolution of singularities for $\cX\to \mathbb D$.
\end{enumerate} 

\medskip
\noindent By Proposition~\ref{quanorbi} and Remark \ref{rem:more_quantitative} that follows, we infer that via a local uniformization of $X_{t_k}$, the pull-back of $\displaystyle \omega_{t_k}$ is quasi-isometric with the Euclidean metric, with a constant which is uniform with respect to $k$. By Lemma~\ref{ke conv}, this implies that for any $U\subset X^{\rm orb}$ admitting a uniformizing map $p:V\to U$ as in \eqref{Galois cover}, we have 
\begin{equation}
\label{laplacian}
C^{-1} \omega_V \le p^*({\omega_0}|_U) \le C \omega_V
\end{equation}
as currents on $V$ for some constant $C>0$. To upgrade the above laplacian estimate into higher order ones and see that $p^*\omega|_U$ is indeed a smooth Kähler metric, one can rely on the general version of Evans-Krylov's $C^{2,\alpha}$-estimate proved in \cite{WangY}.
\end{proof}

\noindent As a consequence of the previous theorem, we obtain the next statement.

\begin{cor}\label{orbis}
Let $(X, \omega_X)$ be a compact Kähler space with log terminal singularities such that $c_1(K_X)= 0.$
Let $\omega= \omega_X+ dd^c\varphi$ be the unique Ricci-flat K\"ahler metric in $[\omega_X]$. 

\noindent
Then the restriction $\displaystyle 
\omega|_{X^{\rm orb}}$ has orbifold singularities. 
\end{cor}

\begin{proof}
Indeed, thanks to \cite[Theorem B]{BGL}, we know that $X$ admits a locally trivial deformation with arbitrarily close algebraic fibers. 
\end{proof}

\subsection{Metric completion near the orbifold locus}

In this last paragraph, we would like to explain a relatively elementary application of Corollary~\ref{orbis} in terms of metric geometry. 

Let $X$ be an $n$-dimensional complex Kähler orbifold (i.e. a normal Kähler variety with at most quotient singularities), let $\omega$ be an orbifold Kähler metric on $X=X^{\rm orb}$ and let $\omega_X$ be a Kähler metric in the sense of Definition~\ref{def1}. For $a,b \in X_{\rm reg}$, define 
\[
d_\om(a,b) = \inf\{\text{length}_\omega(\gamma) \mid \gamma \text{ is a smooth curve in $X_{\rm reg}$ joining $a$ and $b$}\}.
\]
and define $d_X$ similarly. Both functions can be extended to $X\times X$ by considering paths which meet $X_{\rm sing}$ at most at their endpoints.

\begin{prop}\label{prop:orb_metric_completion}
Let $K$ be a compact subset in $X$ with non-empty, connected interior. Then, there exist $C>0$ and $\alpha\in (0,1)$ such that the following inequality holds on $K\times K$.
\[C^{-1}d_X \le d_\omega\le C d_X^{\alpha}.\]
In particular, the metric completion $\overline{(K \setminus X_{\rm sing}, d_\omega)}$ is homeomorphic to $K$.
\end{prop}

\begin{proof}
Let us first explain why the topological identification of the metric completion follows from the distance estimates. The identity map induces maps $F: (K \setminus X_{\rm sing}, d_\omega) \to (K,d_{\omega_X})$ (resp. $G: (K \setminus X_{\rm sing}, d_X) \to (K,d_{\omega})$) which are Lipschitz continuous (resp. Hölder continuous), hence extend uniquely to the metric completions in a continuous way. Moreover, $F$ and $G$ are inverses of each other, hence the same holds for their respective extension, showing the claim. 

Let us now establish the distance estimates. Since $\omega$ is an orbifold metric, it follows immediately that there is a constant $c_K > 0$ such that 
\begin{equation}\label{eq:strict_pos}
	\omega \geq c_K \omega_X
\end{equation} 
as currents on the interior of $K$. In particular, we get $d_\omega \ge c_K^{\frac 12} d_X$ on $K$, establishing the first inequality. To show the reverse inequality $d_\omega \le Cd_X^{\alpha}$, it is enough to work locally in a neighborhood $U$ of a given point $x_0\in K$ (or, to be more precise, we work on $U\times U \subset X\times X$). Moreover, one can assume without loss of generality that there is a uniformizing cover $p:V\to U$ where $V = B_{\C^n}(0,1) \subset \C^n$ and $G = {\rm Gal}(p)$ is a finite subgroup of $\mathrm{U}(n)$. Set $\widetilde{\omega} = p^\ast \omega$.

%

Pick a set $(f_j)_{j=1,\cdots,N}$ of polynomials generating $\C[x_1,\cdots,x_n]^G$. This induces an embedding $U \hookrightarrow \C^N$, fitting into the commutative diagram 
\[
\begin{tikzcd}
	V \ar[d,"p"']\ar[dr,"f"]& \\
	U \ar[r, hookrightarrow]& \C^N
\end{tikzcd}
\]
where $f = (f_1,\cdots,f_N)$.  
	
\smallskip
By the orbifold regularity of $\omega$, there is a $C>0$ such that $\widetilde{\omega} \leq C^2 \om_{\C^n}$ which implies that $d_{\widetilde{\omega}} \leq C d_{\omega_{\C^n}}$ on $V$. 
Since $\om_{\C^n}$ is $G$-invariant, it descend to a metric on $U_{\rm reg}$. 
We denote the corresponding induced distance by $d_{\rm flat}$. 
Up to shrinking $U$, we also obtain $d_\om \leq C d_{\rm flat}$ and $d_{\C^N} \leq C d_{\om_X}$ on $U$. 
For any $\bar{x},\bar{y} \in U_\reg$, let $x,y \in V$ be chosen lifts. Then one can express $d_{\rm flat}(\bar{x},\bar{y}) = \min_{g \in G} \|g\cdot x - y\|$. 
	
\smallskip
It remains to show there exist constants $C>0$ and $\alpha>0$ such that 
\[
	d_{\rm flat} \leq C d_{\rm \C^N}^\alpha
\] 
on $U$. 
Equivalently, it suffices to establish that
\begin{equation}\label{eq:loc_redu_1}
	\min_{g \in G} \|g\cdot x - y\| \leq C \|f(x)-f(y)\|^\alpha
\end{equation} 
on $V$ for some uniform $\alpha > 0$ and $C> 0$. 
One can observe that as $V$ is bounded, if the following holds
\begin{equation}\label{eq:loc_redu_2}
	\prod_{g \in G} \|g\cdot x - y\| \leq C \|f(x)-f(y)\|^\beta
\end{equation} 
on $V$, for some uniform $\beta > 0$ and $C > 0$, then \eqref{eq:loc_redu_1} holds with $\alpha=\beta/|G|$.  
	
\smallskip
For each $g \in G$ and $i = 1,\cdots,n$, consider the holomorphic function $h_{g,i}(x,y) := (g\cdot x)_i-y_i$. 
It vanishes on the set $(f(x)-f(y)=0) = V((f_j(x)-f_j(y))_{j=1,\cdots,N})$. 
By Hilbert's Nullstellensatz, there is an $r \in \N^\ast$ such that $h_{g,i}(x,y)^r = \sum_{j=1}^N a_j^{(g,i)} (f_j(x)-f_j(y))$. 
Taking the least common multiple of the exponents, we may assume that such a property holds for all $g$ and $i$ with the same $r$. 
It follows that 
\[
	|(g\cdot x)_i - y_i| \leq C \left(\sum_{j=1}^N |f_j(x) - f_j(y)| \right)^{1/r} 
	\leq C \left(N \sum_{j=1}^N |f_j(x) - f_j(y)|^2 \right)^{1/2r}. 
\] 
Consequently,
\[
	\prod_{g \in G} \|g\cdot x - y\| 
	= \prod_{g \in G} \left(\sum_{i=1}^n |(gx)_i - y_i|^2\right)^{1/2}
	\leq C' \|f(x)-f(y)\|^{|G|/r}
\]
which establishes \eqref{eq:loc_redu_2} and thus completes the proof. 
\end{proof}

Combining Corollary~\ref{bigcoro:orb_CY} with Proposition~\ref{prop:orb_metric_completion}, we obtain the following immediate consequence. 

\begin{cor}
\label{completion}
Let $(X,\omega_X)$ be a compact Kähler space with log terminal singularities and $c_1(X)=0$, and let $\omega$ be the unique singular Ricci-flat Kähler metric in $[\om_X]$. 
For any open set $U$ containing the non-orbifold locus $X \setminus X^{\rm orb}$, set $K := X \setminus U$. Then the metric completion $\overline{(K \setminus X_{\rm sing},\om)}$ is homeomorphic to $K$. 
\end{cor}

\bibliographystyle{smfalpha}
\bibliography{biblio}

\end{document}